\documentclass[11pt,reqno]{amsart}
\usepackage{graphicx,amscd}
\usepackage{amsfonts,amsmath,amssymb}
\newcommand{\rl}{{\mathbb{R}}}
\newcommand{\cx}{{\mathbb{C}}}

\newcommand{\e}{\varepsilon}
\newcommand{\lam}{\lambda}

\renewcommand {\r}{\rho}
\renewcommand{\bar}{\overline}
\let\del=\partial
\newtheorem{thm}{Theorem}
\newtheorem{propos}[thm]{Proposition}
\newtheorem{lem}[thm]{Lemma}

\theoremstyle{definition}

\title{Critical sets of proper holomorphic mappings}
\author {Sergey Pinchuk}
\address{Department of Mathematics, Indiana University, Rawles Hall
831 East 3rd St Bloomington, IN 47405, USA}
\email{pinchuk@indiana.edu}
\author {Rasul Shafikov}
\email{}
\address{Department of Mathematics, The University of Western Ontario, London, Ontario N6A 5B7 Canada}
\email{shafikov@uwo.ca}

\begin{document}



\date{October 3, 2012}

\begin{abstract}
It is shown that if a proper holomorphic map $f: \mathbb C^n  \to \mathbb C^N$,
$1<n\le N$, sends a pseudoconvex real analytic hypersurface of finite type into another
such hypersurface, then any $n-1$ dimensional component of the critical locus of
$f$ intersects both sides of $M$. We apply this result to the problem of boundary regularity 
of proper holomorphic mappings between bounded domains in $\mathbb C^n$.
\end{abstract}

\maketitle

\section{Introduction and main results}

The goal of this article is to prove the following theorem that describes geometry of the critical set
of a proper holomorphic map between real analytic hypersurfaces.

\begin{thm}\label{t.1}
Let $D\subset \cx^n$, $D'\subset \cx^N$, $2\le n \le N$, be domains and $f: D \to D'$
be a proper holomorphic map that extends holomorphically to a neighbourhood  $U\subset \cx^n$
of a point $a\in \partial D$. Suppose that $\partial D \cap U$ and $\partial D' \cap U'$ are
smooth real analytic pseudoconvex hypersurfaces of finite type, where $U'\subset \cx^N$
is a neighbourhood of $f(a)\in \partial D'$. Let $E$ be an irreducible $(n-1)$-dimensional component 
of the critical set of $f$ in $U$ with $a\in E$. Then $E \cap (D\cap U)\ne \varnothing$.
\end{thm}

We note that the neighbourhood $U\ni a$ in Theorem~\ref{t.1} for which $E \cap (D\cap U)\ne \varnothing$
is arbitrarily small. In this case we say that $E$ {\it enters} the domain $D$ at the point $a$. 

We apply  Theorem~\ref{t.1} to the study of the old conjecture that a proper holomorphic map $f: D \to D'$  between 
bounded domains in $\mathbb C^n$ with real analytic boundaries extends holomorphically to a neighbourhood of the closure 
of $D$. The history of this conjecture began in the 70-ties when it was proved for strictly pseudoconvex domains by 
Lewy~\cite{l} and Pinchuk~\cite{p}. The conjecture has been studied by many authors but still remains open in full generality. 
However, it has been proved in the following considerable special cases:   
\begin{enumerate}
\item[(1)]  $D, D’$ are pseudoconvex, $n \ge 2$  (Diederich-Fornaess~\cite{DF88}, Baouendi-Rothchild~\cite{br});
\item[(2)]  $n=2$ (Diederich-Pinchuk~\cite{DiPi3});
\item[(3)]  $f$ is continuous in the closure of $D$, $n\ge 2$ (Diederich-Pinchuk~\cite{DiPi}).
\end{enumerate}
The proofs of these results consist of two major steps. Step 1 is to show that $f$ extends as a holomorphic correspondence to a 
neighbourhood  of the closure of $D$. Step 2 is to prove that this correspondence is, in fact, a holomorphic map.
The main method for step 1 is the multidimensional reflection principle, based on the technique of Segre varieties. 
For a survey on the subject we refer the reader to~\cite{DiPi2}. Except the case $n=2$, step 1 was realized so far only under 
additional assumption of some a priori boundary regularity of $f$. In particular, in~\cite{DiPi} it was proved provided that 
$f \in C(\overline D)$. We also note that continuous extension of $f$ to $\overline D$ was proved in pseudoconvex case by 
Diederich-Fornaess~\cite{DF79}. Step 2 is essentially the following result.

\begin{thm}\label{t.m}
Let  $D, D' \subset \mathbb C^n$, $n\ge 2$,  be bounded domains with real-analytic boundaries and $f: D \to D'$  be a proper holomorphic 
map that extends as a holomorphic correspondence to a neighbourhood of $\overline D$. Then $f$ extends holomorphically to a 
(possibly smaller) neighbourhood  of $\overline D$.
\end{thm}

Theorem~\ref{t.m} and its generalizations have been proved in~\cite{DiPi0},~\cite{DP2004},~\cite{pv} and strongly rely on the 
proof in the case when both domains are pseudoconvex. The key for proving Theorem 2 in the pseudoconvex case is the 
$C^\infty$-smooth extension of $f$ to the closure of $D$  ( see, for instance,~\cite{bell, bc}). However, the existing proof of the  
$C^\infty$ extension is based on very technical and complicated subelliptic estimates for $\overline\partial$-Neumann operator~
\cite{k}. Here we use Theorem~\ref{t.1} to present a more elementary self-contained proof of Theorem~\ref{t.m} in general situation.  
This allows us to simplify previous proofs of the results discussed above by avoiding the use of sophisticated 
$\overline\partial$-machinery. In fact, while Theorem~\ref{t.m} is stated for simplicity as a global result, we prove a local version of it.

\section{Background: Segre varieties, the Segre map and its critical locus.}
Let $M$ be a smooth real analytic hypersurface in $\mathbb C^n$ passing through the 
origin. In a suitable coordinate system we may assume that it is given by a defining function
$$
\r(z,\bar z) = z_n + \bar z_n + \sum_{|j|,|k|>0} a_{jk}(y_n)\, 'z^j \,'\bar z^k,
$$
where $'z=(z_1, \dots ,z_{n-1})$. By the Implicit Function Theorem, the complexified equation 
$\r(z, \bar w)=0$ can be solved for $z_n$:
\begin{equation}\label{e.lam}
z_n = -{\bar w}_n + \sum_k \bar{\lam_k(w)}\,'z^k,\ \ k = (k_1, \dots, k_{n-1}).
\end{equation}
The Segre varieties are defined as $Q_w = \{z: \r(z,\bar w)=0\}$, and $M$ is called {\it essentially
finite} at zero, if the Segre map $\lam : w \to Q_w$ is finite in a neighbourhood of the origin.
The Segre map can be identified with the holomorphic map $\lambda(w) = \{\lam_k (w) \}$, where
$\lam_k$ are the components of the sum in~\eqref{e.lam}. In fact, if $M$ is essentially finite at
zero, then there exists $m>0$ such that 
$$
Q_w = Q_{\tilde w} \ \ \Longleftrightarrow \ \ \lam_k(w) = \lam_k(\tilde w), \ |k| \le m ,
$$
see \cite{DiFo} or \cite{DiPi2} for the proof. Hence, we may identify the Segre map $\lam$ with
a holomorphic map from a neighbourhood of the origin in $\mathbb C^n$ into $\mathbb C^N$,
for some $N>0$, given by 
$$
\lam(w) = \{\lam_k(w), \ |k| \le m\}.
$$

A smooth real hypersurface $M$ is of {\it finite type} (in the sense of D'Angelo) at a point $p\in M$, 
if the order of contact of $M$ with any one-dimensional complex analytic set passing through $p$ is
bounded above. If $M$ is real analytic, then $M$ is of finite type at $p$ if and only if there does not
exist a germ at $p$ of a positive dimensional analytic set  contained in $M$. In particular, this means that 
$M$ is essentially finite near $p$, and so the Segre map if finite.

\section{Proof of Theorem~\ref{t.1}}

\begin{proof}[Proof of Theorem~\ref{t.1}]
Without loss of generality we may assume that $a=0$, $f(0)=0'$, and $f(U)\subset U'$.
Clearly, $f(D\cap U) \subset D' \cap U'$ and $f(\partial D \cap U) \subset \del D' \cap U'$. 
By the result of Diederich and Fornaess \cite{DF}, for any $\e>0$  in a sufficiently small 
neighbourhood $U'$ of the origin the hypersurface $\del D' \cap U'$ admits a defining function 
$\rho'\in C^2(U')$ such that $\phi':=-(-\rho')^{1-\e}$ is a plurisubharmonic function on 
$D' \cap U'$. It follows that $\phi' \circ f$ is a negative plurisubharmonic function in $D\cap U$,
and so by the Hopf lemma there exists a constant $C>0$ such that
\begin{equation}\label{e.hopf}
|\phi' \circ f (z)| \ge C {\rm dist\,}(z, \del D), \ \  z\in \del D \cap U .
\end{equation}
Throughout the paper ${\rm dist\,}(\cdot, \cdot)$ denotes the usual Euclidean distance between
sets in a Euclidean space.
We may assume that complex tangents to $\del D$ and $\del D'$ at $0$ and $0'$ are given 
respectively by $\{z_n=0\}$ and $\{z_N=0\}$. Then it 
follows from~\eqref{e.hopf} that 
\begin{equation}\label{e.norm0}
\frac{\del f_N}{\del z_n}(0)\ne 0 .
\end{equation} 
Indeed, if otherwise
$\frac{\del f_N}{\del z_n}(0)= 0$, then $f_N(z)=O(|z|^2)$, and since $\rho'(z')=2x'_N + O(|z'|^2)$, 
we obtain
 $$
|\phi'\circ f(z)| \le c_1 |z|^{2(1-\e)}, 
 $$
 which contradicts~\eqref{e.hopf} for $\e<1/2$. In particular, we conclude that the map $f$ extends 
 to $U$ as a proper holomorphic map. This can be seen as follows:~\eqref{e.norm0} implies that $f(U\setminus D) \subset U'$,
 and therefore, $f^{-1}(\partial D \cap U') \subset \partial D$. Since $\partial D$ is of finite type, the set
 $f^{-1}(0')$ is discrete, and, after shrinking if necessary the neighbourhood $U$, we may assume that the
 map $f$ is proper in $U$. 
 
By Remmert's proper mapping theorem  $E' = f(E)\subset U'$ is an irreducible analytic set of dimension $n-1$. 
To illustrate the idea of the proof of the theorem consider first the simple case 
 when  $E$ and $E'$ are complex manifolds. Arguing by contradiction suppose that 
$E \cap (D\cap U) =\varnothing$ for arbitrarily small $U$. Then $E$ is tangent to $\del D$ at the origin.
Since $E' \cap (D' \cap U')$ is also empty, the manifold $E'$ is tangent to $\del D'$ at $0'$. After an 
additional local biholomorphic change of coordinates we may assume that $E=\{z_n=0\}$ and 
$E'=\{z'_N=0\}$. Let $z=(\tilde z, z_n)$, $z'=(\tilde z', z'_N)$, and $f=(\tilde f, f_N)$. Then the
restriction $f|_E$ is given by $\tilde z'=\tilde f(\tilde z, 0)$. Since $f$ is proper at the origin, $f|_E$ is
also proper at~$0$, and therefore the rank of the Jacobian matrix $\frac{\del \tilde f}{\del \tilde z}(\tilde z,0)$
is equal to $n-1$ on a dense subset $E_1\subset E$. On the other hand, ${\rm rank} \frac{\del f}{\del z}<n$
for $z=(\tilde z,0)$, and $\frac{\del f_n}{\del z_j}(\tilde z,0)=0$, $j=1,\dots, n-1$, because $f_n(\tilde z,0)=0$.
Therefore, $\frac{\del f_N}{\del z_n}(\tilde z,0) =0$ for $(\tilde z,0) \in E_1$. By  continuity, 
\begin{equation}\label{e.der0}
\frac{\del f_N}{\del z_n}(0)=0,
\end{equation}
which contradicts \eqref{e.norm0}.

For the proof in the general case we will need the following technical result. We denote by ${\rm reg\,} E$ the 
locus of regular points of a complex analytic set $E$, i.e., the points near which $E$ is locally a complex manifold.
Then ${\rm sing\,} E = E \setminus {\rm reg\,} E$ is the singular locus of $E$.

\begin{propos}\label{p.1}
There exist a sequence of points $\{p^\nu\}\subset {\rm reg\,} E$ and two sequences of complex affine maps
$A^\nu :\cx^n \to \cx^n$, $B^\nu: \cx^N \to \cx^N$ such that for every $\nu=1,2,\dots$, the 
following holds
\begin{enumerate}
\item[(i)] ${\rm rank}(f|_E) = n-1$ at $p^\nu$, and $f(p^\nu) \in {\rm reg\,} E'$.
\item[(ii)] $A^\nu(p^\nu)=p^\nu$ and $B^\nu(f(p^\nu)) = f(p^\nu)$.
\item[(iii)] The transformations $A^\nu$, $B^\nu$ converge to the identity maps $I_n: \cx^n \to \cx^n$
and $I_N: \cx^N \to \cx^N$ respectively.
\item[(iv)] $dA^\nu$ maps $T_{p^\nu}E$ onto $\{v \in T_{p^\nu}\cx^n: v_n =0\}$ and $dB^\nu$ maps
$T_{f(p^\nu)}E'$ onto $\{v \in T_{f(p^\nu)}\cx^N: v_N =0\}$.
\end{enumerate}
\end{propos}

Theorem~\ref{t.1} can be easily deduced from Proposition~\ref{p.1}. Indeed, consider the 
sequence of maps 
$f^\nu = B^\nu \circ f \circ (A^\nu)^{-1}$. The above arguments show that 
$$
\frac{\del f^\nu_N}{\del z_n}(p^\nu) \to 0 \ \ {\rm as \ } \nu\to\infty,
$$
which yields~\eqref{e.der0}. Again, we obtain a contradiction with the Hopf lemma.
\end{proof}

The rest of the section is devoted to the proof of Proposition~\ref{p.1}. We will need the following

\begin{lem}\label{l.one}
Let $U\subset \cx^n$ be a neighbourhood of the origin, $M\ni 0$ be a real hypersurface
in $U$ with a defining function $\rho \in C^1(U)$, 
\begin{equation}\label{e.r1}
\rho(z)=2x_n + o(|z|) .
\end{equation}
Let $A\subset U$ be an analytic set of pure dimension $d$, $1\le d < n$, such that $0\in A \subset \{z\in U :
\rho(z) \ge 0\}$. Then there exists an open subset $V\subset {\rm reg\,}A$ with $0\in \overline V$
such that for any point $p\in V$ the tangent plane $T_p A$ is contained in a complex hyperplane
$$
L_p = \{v\in\cx^n: v_n = \sum_{k=1}^{n-1} a_k(p) v_k\},
$$
and $\lim_{V\ni p \to 0} a_k(p)=0$ for any $k=1,2,\dots, n-1$.
\end{lem}

\begin{proof}
Let $C_0(A)$ be the tangent cone of $A$ at $0$. It is defined by $C_0(A)= \lim_{t\to 0} A_t$, where 
$A_t = \{ tz: z\in A\}$, $t\in \rl_+$, are isotropic dilations of $A$. The set $C_0(A)$ is a complex
cone of dimension $d$, i.e., it is invariant under complex dilations $z\to tz$, $t\in \cx \setminus \{0\}$
(see, e.g.,~\cite{Ch}) and $0\in C_0(A) \subset\{z_n \ge 0\}$. The last inclusion follows from
$A_t \subset \{z: t \rho(z/t) \ge 0\}$ and $t\rho(z/t) \to 2x_n$ as $t\to \infty$ because of~\eqref{e.r1}.
By the maximum principle we conclude that 
\begin{equation}\label{e.2}
C_0(A) \subset \{z_n =0 \}.
\end{equation}
Since $\dim C_0(A)=d$, there exists a complex plane $L\ni 0$, $\dim L = n-d$, such that $L \cap C_0(A)= \{0\}$.
Without loss of generality we assume that 
\begin{equation}\label{e.3}
L = \{z\in \cx^n : z_1=0, \dots, z_d = 0\} .
\end{equation}
Let $\tilde z = (z_1, \dots, z_d)$, $\tilde{\tilde z} =  (z_{d+1}, \dots, z_{n-1})$ so that 
$z = (\tilde z, \tilde{\tilde z}, z_n)$. It follows from~\eqref{e.2} that $|z_n| = o(|\tilde z| + |\tilde{\tilde z}|)$
on $A$, i.e., there exists a continuous function $\alpha(t) \ge 0$ for $t\ge 0$ such that 
\begin{equation}\label{e.four}
|z_n| \le \alpha(|\tilde z| + |\tilde{\tilde z}|) (|\tilde z| + |\tilde{\tilde z}|), \ \ z\in A .
\end{equation}
We also have the following estimate for some $c_1>0$ and all $z\in C_0(A)$:
\begin{equation}\label{e.5}
|z_n| + |\tilde{\tilde z}| \le c_1|\tilde z|,
\end{equation}
which follows from $L \cap C_0(A) = \{0\}$ and~\eqref{e.3}. This implies that the origin is an isolated
point of $L\cap A$. Hence,~\eqref{e.5} also holds for $z\in A$, possibly with a different $c_1$. 

Now we can choose 
$$
U = \tilde U \times \tilde{\tilde U} \times U_n \subset \cx^d \times \cx^{n-d-1} \times \cx
$$
such that $\pi: A\cap U \to \tilde U$ is a branched analytic covering of some multiplicity $m\ge 1$. Its 
discriminant set $\tilde \sigma \subset \tilde U$ and the tangent cone $C_0(\tilde \sigma) \subset \cx^d$
are analytic sets of dimension at most $d-1$. Therefore, there exists a complex line $\tilde l \subset \cx^d$
such that $C_0(\tilde \sigma) \cap \tilde l = \{0\}$. We may assume that 
$\tilde l = \{(z_1,0,\dots,0) \in \cx^d : z_1\in \cx\}$. 
Since $C_0(\tilde \sigma)$ is a closed cone, there exists $\delta >0$ such that
\begin{equation}\label{e.6}
\{\tilde z \in \cx^d : |z_j| < \delta|z_1|, \ j=2,\dots, d\}\cap C_0(\tilde \sigma) = \varnothing .
\end{equation}
With possibly smaller $\delta>0$ we also have 
\begin{equation}\label{e.7}
\{ \tilde z \in \tilde U : |z_j| < \delta |z_1|, \ j=2,\dots, d\}\cap \tilde \sigma = \varnothing .
\end{equation}
The set 
$$
\tilde V_\delta :=\{ \tilde z \in \tilde U : |z_j| < \delta |z_1|,\ j=2,\dots, d\}\cap \{\tilde z \in \tilde U: {\rm Re\,}z_1>0\}
$$
is simply connected, open in $\tilde U$ and contains the origin in its closure. Since $\tilde V_\delta \cap \tilde \sigma =
\varnothing$ the set $A \cap (\tilde V_\delta \times \tilde{\tilde U} \times U_n)$ is the union of the graphs of 
$m$ holomorphic mappings $\tilde V_\delta \to \tilde{\tilde U} \times U_n$. Consider one of them, $H=(\tilde{\tilde h}, h_n)$,
and let $A_\delta = A \cap (\tilde V_\delta \times \tilde{\tilde U} \times U_n)$ be its graph. For any 
$p=(\tilde p, \tilde{\tilde p},p_n) \in A_\delta$ the tangent plane $T_p A$ is contained in the tangent plane at $p$ to 
the hypersurface in $\tilde V_\delta \times \tilde{\tilde U} \times U_n$ defined by one equation $z_n = h_n (\tilde z)$,
which is given by 
$$
\left\{v \in \cx^n : v_n = \sum_{k=1}^d a_k(\tilde p) v_k \right\}, \ \  a_k(\tilde p) = \frac{\del h_n}{\del z_k}(\tilde p).
$$
Thus, to finish the proof of the lemma, it is sufficient to show that
\begin{equation}\label{e.eight}
\lim_{\tilde V_\delta \ni \tilde p \to 0} \frac{\del h_n}{\del z_k} (\tilde p) = 0, \ \ k=1,\dots,d .
\end{equation}
Using~\eqref{e.four}--\eqref{e.7} we successively obtain for certain constants $c_j>0$ and all $\tilde p \in V_\delta$,
with $\delta <<1$, the following estimates:
\begin{eqnarray*}
|p_1| \le |\tilde p| \le c_1 |p_1|,\\
{\rm dist\ }(\tilde p, c_0(\tilde\sigma)) \ge c_2|\tilde p| \ge c_2 |p_1|,\\
{\rm dist\ } (\tilde p, \tilde \sigma) \ge c_3 |p_1|.
\end{eqnarray*}
If $B(\tilde p, \tilde \sigma)$ denotes the ball $\{\tilde z \in \mathbb C^d : |\tilde z - \tilde p| < r\}$,
then $B(\tilde p, c_4 |p_1|)\subset \tilde V_\delta$, and $|\tilde z| \le c_5 |p_1|$ for all 
$\tilde z\in B(\tilde p, c_4 |p_1|)$. For $z\in A$ with $\tilde z \in B(\tilde p, c_4 |p_1|)$ we have
$$
|h_n(\tilde z)| = |z_n| \le \alpha \left(|\tilde z| + |\tilde{\tilde z} |\right)\left(|\tilde z| + |\tilde{\tilde z}|\right) \le
c_5\,\alpha\left(c_5|\tilde z|\right) \,|\tilde z| \le c_6\, \alpha \left(c_6 \alpha|p_1|\right)\, |p_1|.
$$
Now by the Schwarz lemma applied to $h_n(\tilde z)$ in $B(\tilde p, c_4 |p_1|)$ we get
$$
\left| \frac{\partial h_n}{\partial z_k}(\tilde p) \right| \le c_7\, \alpha (c_6 |p_1|) ,
$$
and \eqref{e.eight} follows from $\lim_{t\to 0^+} \alpha(t) =0$.
\end{proof}

\begin{proof}[Proof of Proposition~\ref{p.1}.]
The set
$$
E_1:=\left\{z\in {\rm reg\,}E: {\rm \ rank\,}(f|_E) < n-1 {\rm\  at\ } z\right\} \cup {\rm sing\,}E
$$
is nowhere dense and closed in $E$. Therefore, $E'_1: = f(E)$ is closed and nowhere dense in $E'$.
By Lemma~\ref{l.one} with $A=E'$ and $M=\partial D'$ there exist a sequence $p'^\nu \in {\rm reg\,}E'$
and a sequence $p^\nu \in {\rm reg\,}E$ such that
\begin{itemize}
\item[(a)] $p'^\nu = f(p^\nu)$,
\item[(b)] $\lim_\nu p^\nu = 0$, $\lim_\nu p'^\nu = 0'$,
\item[(c)] ${\rm rank\,}(f|_E) = n-1$ at each $p^\nu$,
\item[(d)] for every $\nu$,
\begin{equation}\label{e.nine}
T_{p'^\nu} E' \subset \left\{v \in \mathbb C^N : v'_N = \sum_{k=1}^{N-1} a'_{k\nu}\,v'_{k} \right\}
\end{equation}
and 
\begin{equation}\label{e.ten}
\lim_{\nu\to\infty} a'_{k\nu} = 0, {\rm \ for\ any\ \ } k=1,\dots,N-1.
\end{equation}
\end{itemize}
We claim that 
\begin{equation}\label{e.eleven}
T_{p^\nu} E \subset \left\{v \in \mathbb C^n : v_n = \sum_{k=1}^{n-1} a_{k\nu}\,v_{k} \right\}
\end{equation}
with 
\begin{equation}\label{e.twelve}
\lim_{\nu\to\infty} a_{k\nu} = 0,\ \  k=1,\dots, n-1 .
\end{equation}
Since $f$ is holomorphic near the origin and sends
$\partial D$ into $\partial D'$, the last component $f_N$ of $f$ is of the form
\begin{equation}\label{e.thirteen}
f_N(z) = \mu\, z_N + o(|z|),
\end{equation}
where $\mu \ne 0$ by the Hopf lemma. The equations of $T_{p^\nu}$ can be obtained from 
$df_{p^\nu}(T_{p^\nu}E) \subset T_{p'^\nu} E'$. Using~\eqref{e.nine},~\eqref{e.ten}, and~\eqref{e.thirteen}
we conclude that $T_{p^\nu}E$ are of the form~\eqref{e.eleven} and the coefficients $a_{k\nu}$ satisfy~\eqref{e.twelve}
because of~\eqref{e.ten} and~\eqref{e.thirteen}. The transformations $A^\nu$ and $B^\nu$ can be defined by
\begin{eqnarray*}
A^\nu: (z_1, \dots, z_{n-1}, z_n) \mapsto \left(z_1, \dots, z_{n-1}, z_n - 
\sum_{k=1}^{n-1} a_{k\nu}(z_k - p_k^\nu)\right),\\
B^\nu: (z'_1,\dots, z'_{N-1}, z'_N) \mapsto 
\left(z'_1,\dots, z'_{N-1}, z'_N - \sum_{k=1}^{N_1} a'_{k\nu} (z'_k - p'^\nu_k) \right) .
\end{eqnarray*}
They satisfy the required properties, and this completes the proof of the proposition and Theorem~\ref{t.1}.
\end{proof}

\section{Proof of Theorem~\ref{t.m}}

The crucial step in the proof of Theorem~\ref{t.m} is the following 

\begin{lem}\label{l.3}
If in the situation of Theorem~\ref{t.m} every irreducible component $E\ni a$ of a branch
locus of the correspondence that extends $f$ enters $D$ at $a$, then $f$ extends holomorphically
to $a$.
\end{lem}

\begin{proof} 
Let $U$ be a small neighbourhood of $a$, and $F: U \to \mathbb C^n$ be the correspondence that 
extends the map $f$ near the point $a$. Let $E$ be the branch locus of $F$ in $U$. Then $E$ is a
complex analytic set of pure dimension $n-1$. Since every component of $E$ enters the domain $D$ at $a$, 
we may choose the neighbourhood $U$ so small that for every irreducible component $\tilde E$ of $E$,
the set $\tilde E \cap D$ is nonempty and open in $\tilde E$. 

Let $S = E \setminus D$. We claim that $U\setminus S$ is simply connected. For the proof we will show 
that every nontrivial cycle in $U\setminus E$ is null-homotopic in $U\setminus S$, from this simple connectivity 
of $U\setminus S$ follows. By the classical van Kampen-Zariski Theorem, see, e.g., \cite{Dim}, the fundamental 
group of $U\setminus E$ is generated by the cycles that generate the fundamental group of $L \setminus (E \cap L)$, 
where $L$ is a complex line intersecting $E$ transversely and avoiding singular points of $E$. Let $\gamma$ be a 
generator of $\pi_1(L \setminus (E \cap L))$. Then $\gamma$ is homotopic to a small circle in $L$ around a point 
$p$ of the intersection of $L$ with an irreducible component $\tilde E$ of $E$. Further, the point $p$ is a regular 
point of $\tilde E$, and $\gamma \cap E = \varnothing$. Since the locus of regular points of $\tilde E$ is connected 
and $\tilde E \cap D$ contains an open subset of $\tilde E$ by the assumptions of the lemma, we can move the cycle 
$\gamma$ along the locus of smooth points of $\tilde E$ avoiding points in $E$ until $\gamma$ is entirely contained 
in $D$. This means that $\gamma$ is null-homotopic in $U\setminus S$, and hence the latter is simply connected.

We next show that the map $f$ defined in $D\cap U$ extends as a holomorphic map along any path in 
$U \setminus S$. Indeed, on $U\setminus E$ the correspondence $F$ splits into a finite collection of holomorphic 
mappings, the branches of $F$. Fix a point $b\in  (U \cap\partial D) \setminus E$. Then one of the branches of the correspondence $F$  at $b$ gives the extension of the map $f$ to a neighbourhood of $b$. Taking any path 
$\gamma$ in $U\setminus E$ which starts at $b$ we obtain the extension of $f$ along~$\gamma$ by choosing 
the appropriate branches of $F$ over the points in $\gamma$. This gives analytic continuation of $f$ in the 
complement of $E$ in $U$. Suppose now that $\gamma$ intersects $E\cap D$. Without loss of generality assume 
that $\gamma$ terminates at a point $c\in E\cap D$ and $\gamma \setminus \{c\} \subset U \setminus E$. The set 
$S$ is closed and has simply connected complement in $U$, hence, any two paths in the complement of $S$ are
homotopically equivalent. In particular, this means that the path $\gamma$ can be homotopically deformed
avoiding the set $S$ so that the deformation $\tilde \gamma$ of $\gamma$ connects the points $b$ and $c$
along the path that is entirely contained in $D \setminus E$ (except the end points). Furthermore, we claim that 
this can be done in such a way that no curve in the deformation family intersects~$E$ (except the end point). 
Indeed, consider the cycle $\gamma \circ \tilde \gamma^{-1}$ which we slightly deform so that it does not 
intersect~$E$ near the point $c$. If $\gamma \circ \tilde \gamma^{-1}$ is null-homotopic in $U\setminus E$, 
then the claim is trivial. If $\gamma \circ \tilde \gamma^{-1}$ is a nontrivial cycle in $U\setminus E$, then as in 
the proof of simple connectivity of $U\setminus S$, we may represent this cycle as a sum of ``small'' cycles around
smooth points of $E$. We then move these small cycles along the regular locus of $E$ until all of them are 
contained in $D$ (again we used the fact that every component of $E$ enters the domain $D$). As a result we
conclude by the Monodromy theorem that the analytic continuation of $f$ along $\gamma$ and $\tilde \gamma$
defines the same analytic element near the point $c$. But since $\tilde \gamma$ is contained in $D$, extension
along $\tilde \gamma$ simply gives the map $f$ already defined at $c$. This gives analytic continuation
of $f$ along any path in $U \setminus S$, which is single-valued by the Monodromy theorem. 

Finally, since every component of $E$ enters the domain $D$
at $a$, the set $S$ is not a complex analytic subset of $U$, and hence it is a removable singularity for the
extension of $f$ in $U\setminus S$. This shows that $f$ extends to $a$ as a holomorphic map.
\end{proof}

\begin{proof}[Proof of Theorem~\ref{t.m}]
Choose normal coordinates near the points $a$, $f(a)$ and assume $a=0$, $f(a)=0'$. By $\rho$ and $\rho'$
we denote local defining functions of $D\cap U$, and $D'\cap U'$ respectively, of the form
\begin{eqnarray}
\label{e.rho}
\rho(z,\bar z) = 2 x_2 + \sum_{|k|,|l|\ge 1} a_{kl}(y_n) \tilde z^k \overline{\tilde z}^l,\\
\label{e.rho'}
\rho'(z',\bar z') = 2 x'_2 + \sum_{|k|,|l|\ge 1} a'_{kl}(y'_n) \tilde z'^k \overline{\tilde z'}^l,
\end{eqnarray}
Let $\lambda : U \to \mathbb C^{N+1}$, $\lambda': U' \to \mathbb C^{N'+1}$ be the Segre maps of
$\partial D$ and $\partial D'$ near $0\in U$ and $0'\in U'$ respectively. It is convenient to denote their
components by $\lambda=(\lambda_0, \lambda_1, \dots, \lambda_N)$, 
$\lambda'=(\lambda'_0, \lambda'_1, \dots, \lambda'_N)$ so that in normal coordinates
\begin{equation}\label{e.tri}
(a)\ \lambda_0(z)=z_n, \ \ (b)\ \lambda'_0 (z') = z'_N.
\end{equation}
We will need some results from~\cite{DiPi3} which can be summarized as follows.

\begin{propos}[Diederich and Pinchuk, \cite{DiPi3}]\label{p.1995}
Let $F: U\to U'$ be the correspondence extending $f: D\cap U \to D'\cap U'$, where $U\ni 0$, $U'\ni 0'$
are small enough. Then
\begin{itemize}
\item[(i)] there exists a single-valued (even injective) map $\phi: \lambda(U) \to \lambda'(U')$ such that
the following diagram commutes.
\begin{equation}\label{e.chetyre}
\begin{CD}
\lambda(U) @>\phi>> \lambda'(U') \\
@VV\lambda V @VV\lambda' V\\
U @>F>> U' .
\end{CD} 
\end{equation}
For the (multiple-valued) correspondence $F$ this means that for any $z\in U$, it commutes with any
value of $F(z)$ (Cor. 4.2 and 5.5 in \cite{DiPi3});

\item[(ii)] $F(D\cap U) \subset D' \cap U'$, $F(\partial D \cap U) \subset \partial D' \cap U'$, 
$F(U\setminus D) \subset U' \setminus D'$ (Prop. 7.1);

\item[(iii)] The map $\lambda' \circ F$ is single-valued and holomorphic in $U$ with 
$\lambda'_0\circ F(z) = b(z) z_n$ and $b(0)\ne 0$ (Prop.~7.2);

\item[(iv)] $F: U \to U'$ is locally proper at the origin, i.e., $F^{-1}(0)=\{0\}$ and therefore,
$F^{-1}$ is also a holomorphic correspondence near $0'$ (Thm~5.1).

\end{itemize}
\end{propos}

We will assume that $b(z)\equiv 1$. This can be achieved by an additional change of coordinates in~$U$.
Of course, these coordinates may no longer be normal. Instead we have 
\begin{equation}\label{e.pyat'}
F_n(z) = f_n(z) = z_n .
\end{equation}
Denote by $\Omega'$ a neighbourhood of $\lambda'(0')$ in $\mathbb C^{N'+1}$. We can choose the sets
$U\ni 0$, $U'\ni 0'$, $\Omega' \ni \lambda'(0')$ such that the mappings $f: D\cap U \to D'\cap U'$ and
$\lambda': U' \to \Omega'$ are proper holomorphic. Consider for $M>0$ the following open sets
\begin{eqnarray*}
D'_M &=& \left\{z'\in U' : \ 2x'_n + M \sum_{k=0}^{N'} \left| \lambda'_k(z')\right|^2 < 0 \right\},\\
D_M &=& \left\{z\in U' : \ 2x_n + M \sum_{k=0}^{N'} \left| \lambda'_k (F(z))\right|^2 < 0 \right\} .
\end{eqnarray*}
The boundaries $\partial D'_M$, $\partial D_M$ near $0'$ and $0$ respectively, are real analytic and 
pseudoconvex because of \eqref{e.tri}(b), and of finite type because of properness of $\lambda'$ and 
Proposition~\ref{p.1995}(iv). 

We first prove Theorem~\ref{t.m} under an additional assumption that $D'_M \cap U' \subset D' \cap U'$.
It follows from Proposition~\ref{p.1995} and~\eqref{e.pyat'} that $D_M \cap U \subset D \cap U$ and
$f:D_M \cap U \to D'_M \cap U'$ is a proper holomorphic map. This implies that for 
$$
\Omega'_M = \{ w \in \Omega' : 2{\rm Re\,}w_0 + M|w|^2 < 0\} ,
$$
the map $\lambda' \circ f : D_M \cap U \to \Omega'_M$ is also proper holomorphic. By Proposition~\ref{p.1995}(iii),
the map $\lambda'\circ F = \lambda' \circ f$ extends holomorphically to a neighbourhood of $0\in U$. 

Let $E'\subset U'$ be the critical set of $\lambda' : U' \to \Omega'$ and $S\subset U$ be the branch locus
of $F: U\to U'$. By Proposition~\ref{p.1995}, $F(S)\subset E'$, moreover, $F(S)$ is contained in the 
$(n-1)$-dimensional part of $E'$. By Theorem~\ref{t.1} any $(n-1)$-dimensional component of $E'$ enters
$D'\cap U'$ at $0'$. By Proposition~\ref{p.1995}(ii), any irreducible component of $S$ also enters $D\cap U$
at $0$, and thus $f$ extends holomorphically to $0$ by Lemma~\ref{l.3}. This completes the proof of Theorem~\ref{t.m}
in the case $D'_M \cap U' \subset D' \cap U'$. However, $D'_M \cap U'$ is not necessarily a subset of $D'\cap U'$
and the general proof of Theorem~\ref{t.m} requires an additional (mainly technical) argument.

As in~\cite{DiPi0}, consider for $M>1$ two families of open sets depending on $\e \in (-\frac{1}{M},0]$:
\begin{eqnarray*}
D'_{M\e} &=& \left\{ z'\in U': \ 2x'_n+M\sum_{k=0}^{N'} |\lambda'_k (z')|^2 < \e \right\},\\
D_{M\e} &=& \left\{ z\in U: \ 2x_n+M\sum_{k=0}^{N'} |\lambda'_k\circ F(z)|^2 < \e \right\},
\end{eqnarray*}
These families are increasing for increasing $\e$ and $D'_{M0} = D'_M$, $D_{M0}=D_M$. The next proposition
summarizes some results in~\cite{DiPi0}.

\begin{propos}[Diederich and Pinchuk, \cite{DiPi0}]\label{p.1998} $\ $
\begin{itemize}
\item[(a)] The sets $D'_{M\e}$, $D_{M\e}$ are pseudoconvex and their boundaries are of
finite type at all points in $U$, respectively $U'$, where they are smooth real analytic.

\item[(b)] $D'_{M\e} \subset D'\cap U$ and $D_{M\e} \subset D \cap U$ if $\e \in (-\frac{1}{M},0]$ is
close to $-\frac{1}{M}$.

\item[(c)] For $M>0$ sufficiently large and any $\e\in (-\frac{1}{M},0]$ the nonsmooth part of $\partial D'_{M\e}$
is contained in $D'\cap U'$ and the nonsmooth part of $\partial D_{M\e}$ is contained in $D \cap U'$.
\end{itemize}
\end{propos}

To finish the proof of Theorem~\ref{t.m} consider
$$
\Omega'_{M\e} = \left\{ w'\in \mathbb C^{N'+1}:\ 2u_0 + M|w'|^2 < \e \right\} .
$$
If $M$, $\e$ are chosen as in Proposition~\ref{p.1998}, then $f: D_{M\e} \to D'_{M\e}$ and 
$\lambda \circ f : D_{M\e} \to \Omega'_{M\e}$ are proper holomorphic maps. Consider the largest 
$\e_0 \in (-\frac{1}{M},0]$ such that $f$ extends to a proper holomorphic map $\tilde f: D_{M\e} \to D'_{M\e_0}$.
By Proposition~\ref{p.1998}, $\tilde f$ is holomorphic on the nonsmooth part of $\partial D_{M\e_0}$. Let us show
that $\tilde f$ extends holomorphically to any smooth real analytic boundary point $a\in U$ of $D_{M\e}$. We only
need to consider the case $a\in S$. Applying, as before, Theorem~\ref{t.1} to the map $\lambda':D'_{M\e_0} \to
\Omega'_{M\e_0}$, we conclude that any irreducible $(n-1)$-dimensional component $E'_j \ni \tilde f(a)$ of the
critical set $E'$ of $\lambda'$ enters $D'_{M\e_0}$ at $\tilde f(a)$. By Proposition~\ref{p.1995}, any irreducible
component $S_j\ni a$ of $S$ enters $D_{M\e_0}$ at $a$. Thus, by Lemma~\ref{l.3}, $\tilde f$ extends 
holomorphically to any such $a$. This means that $f$ extends holomorphically to a neighbourhood of the closure
of $D_{M\e_0}$ and $\e_0=0$. This completes the proof.
\end{proof}



\begin{thebibliography}{99}

\bibitem{br} M.S. Baouendi and L. Rothschild. {\it Germs of CR maps between real analytic hypersurfaces.} 
Invent. Math. {\bf 93} (1988), no. 3, 481-500.

\bibitem{bell} S. Bell. {\it Biholomorphic mappings and the $\overline\partial$-problem.} Ann. of Math. (2) 
{\bf 114} (1981), no. 1, 103-113.

\bibitem{bc} S. Bell and D. Catlin. {\it Boundary regularity of proper holomorphic mappings.} Duke Math. J. 
{\bf 49} (1982), no. 2, 385-396.

\bibitem{Ch} E. Chirka. ``Complex analytic sets." Kluwer, Dordrecht, 1989.

\bibitem{DiFo} K. Diederich and J. E. Fornaess. {\it Proper holomorphic
mappings between real-analytic pseudoconvex domains in ${\mathbb
C}^n$.} Math. Ann. {\bf 282} (1988), 681--700.

\bibitem{DF} K. Diederich and J.E. Fornaess. {\it Pseudoconvex domains: bounded strictly plurisubharmonic 
exhaustion functions.} Invent. Math. {\bf 39} (1977), no. 2, 129-141.

\bibitem{DF78} K. Diederich and J. E. Fornaess. {\it Pseudoconvex domains with real-analytic boundary.} 
Ann. Math. (2) {\bf 107} (1978), no. 2, 371-384.

\bibitem{DF79} K. Diederich and J. E. Fornaess. {\it Proper holomorphic maps onto pseudoconvex domains 
with real-analytic boundary.} Ann. of Math. (2) {\bf 110} (1979), no. 3, 575-592.

\bibitem{DF88} K. Diederich and J. E. Fornaess. {\it Proper holomorphic mappings between real-analytic pseudoconvex domains in 
$\mathbb C^n$.} Math. Ann. {\bf 282} (1988), no. 4, 681-700.

\bibitem{DiPi3} K. Diederich and S. Pinchuk. {\it Proper holomorphic maps in dimension 2 extend}. 
Indiana Univ. Math. J. {\bf 44} (1995), no. 4, 1089--1126.

\bibitem{DiPi0} K. Diederich and S. Pinchuk. {\it Reflection principle in higher dimensions.} 
Proceedings of the International Congress of Mathematicians, Vol. II (Berlin, 1998). Doc. Math. 
1998, Extra Vol. II, 703--712.

\bibitem{DiPi} K. Diederich and S. Pinchuk. {\it Regularity of continuous
CR maps in arbitrary dimension.} Michigan Math. J. {\bf 51}
(2003), no. 1, 111--140.

\bibitem{DP2004} K. Diederich and S. Pinchuk. {\it Analytic sets extending the graphs of holomorphic mappings.} 
J. Geom. Anal. {\bf 14} (2004), no. 2, 231-239.

\bibitem{DiPi2} K. Diederich and S. Pinchuk. {\it The geometric reflection principle
in several complex variables: a survey.} Complex Var. Elliptic
Equ.{\bf 54} (2009), no. 3-4, 223--241.

\bibitem{Dim} A. Dimca. Singularities and Topology of Hypersurfaces. Springer-Verlag, 1992.

\bibitem{k} J. Kohn. {\it Subellipticity of the $\overline\partial$-Neumann problem on pseudo-convex domains: sufficient conditions.} 
Acta Math. {\bf 142} (1979), no. 1-2, 79-122.

\bibitem{l} H. Lewy. {\it On the boundary behaviour of holomorphic mappings.} Acad. Naz. Lincei, {\bf 35} (1977), no 1, 1-8.

\bibitem{p} S. Pinchuk. {\it The analytic continuation of holomorphic mappings.} Mat. Sb. (N.S.) {\bf 98} (140) (1975), no. 3(11), 
416-435, 495-496.

\bibitem{pv} S. Pinchuk and K. Verma. {\it  Analytic sets and the boundary regularity of CR mappings.} 
Proc. Amer. Math. Soc. {\bf 129} (2001), no. 9, 2623-2632.

\bibitem{webster}  S.~Webster. {\it On the mappings problem for algebraic
real hyprsurfaces}, Invent. Math., {\bf 43} (1977), 53--68.

\end{thebibliography}
\end{document}